\documentclass[11pt]{article}

\usepackage{fullpage}
\usepackage{amsmath,amsfonts,amsthm,mathrsfs,mathpazo,xspace,graphicx}
\usepackage{times}
\usepackage{amssymb,latexsym}

\newtheorem{theorem}{Theorem}
\newtheorem{proposition}[theorem]{Proposition}
\newtheorem{lemma}[theorem]{Lemma}
\newtheorem{claim}[theorem]{Claim}

\theoremstyle{remark}

\newcommand{\Tr}{\mbox{\rm Tr}}

\newcommand{\Id}{\ensuremath{\mathop{\rm Id}\nolimits}}

\newcommand{\C}{\ensuremath{\mathbb{C}}}
\newcommand{\N}{\ensuremath{\mathbb{N}}}

\newcommand{\R}{\ensuremath{\mathbb{R}}}

\newcommand{\sch}{\mathsf{S}}
\newcommand{\schs}{\scriptstyle{\mathsf{S}}_1}
\newcommand{\scht}{\scriptstyle{\mathsf{S}}_2}
\newcommand{\schf}{\scriptstyle{\mathsf{S}}_\infty}

\DeclareMathOperator{\sign}{sgn}

\newcommand{\eps}{\varepsilon}

\newif\ifnotes\notesfalse

\newcounter{margincounter}

\ifnotes
\usepackage{color}
\definecolor{mygrey}{gray}{0.50}
\newcommand{\notename}[3]{{\textcolor{#2}{\footnotesize{\bf (#1: \# \arabic{margincounter}\addtocounter{margincounter}{1}} {#3}{\bf ) }}}}
\newcommand{\noteswarning}{{\begin{center} {\Large WARNING: NOTES ON}\end{center}}}

\newcommand{\onote}[1]{{\notename{Oded}{blue}{#1}}}
\newcommand{\tnote}[1]{{\notename{Thomas}{magenta}{#1}}}
\newcommand{\pnote}[1]{\footnote{\notename{Note}{mygrey}{#1}}}

\else

\newcommand{\notename}[2]{{}}
\newcommand{\noteswarning}{{}}

\newcommand{\onote}[1]{}
\newcommand{\tnote}[1]{}
\newcommand{\pnote}[1]{}

\fi

\begin{document}

\title{Bounds on Dimension Reduction in the Nuclear Norm}

\author{
Oded Regev\thanks{Courant Institute of Mathematical Sciences, New York
 University. Supported by the Simons Collaboration on Algorithms and Geometry and by the National Science Foundation (NSF) under Grant No.~CCF-1814524. Any opinions, findings, and conclusions or recommendations expressed in this material are those of the authors and do not necessarily reflect the views of the NSF.}
\and
Thomas Vidick\thanks{Department of Computing and Mathematical Sciences, California Institute of Technology, Pasadena, USA. Supported by NSF CAREER Grant CCF-1553477, a CIFAR Azrieli Global Scholar award, and the Institute for Quantum Information and Matter, an NSF Physics Frontiers Center (NSF Grant PHY-1733907). Email: \texttt{vidick@cms.caltech.edu}.}
}

\maketitle
\noteswarning

\begin{abstract}
For all $n \ge 1$, we give an explicit construction of $m \times m$ matrices $A_1,\ldots,A_n$ with $m = 2^{\lfloor n/2 \rfloor}$ such that for any $d$ and $d \times d$ matrices $A'_1,\ldots,A'_n$ that satisfy 
\[ \|A'_i-A'_j\|_{\schs} \,\leq\, \|A_i-A_j\|_{\schs}\,\leq\, (1+\delta) \|A'_i-A'_j\|_{\schs} \]
for all $i,j\in\{1,\ldots,n\}$ and small enough $\delta = O(n^{-c})$, where $c>0$ is a universal constant, it must be the case that $d \ge 2^{\lfloor n/2\rfloor -1}$. 
This stands in contrast to the metric theory of commutative $\ell_p$ spaces, as it is known that for any $p\geq 1$, any $n$ points in $\ell_p$ embed exactly in $\ell_p^d$ for  $d=n(n-1)/2$.

Our proof is based on matrices derived from a representation of the Clifford algebra generated by $n$ anti-commuting Hermitian matrices that square to identity, and borrows ideas from the analysis of nonlocal games in quantum information theory. 
\end{abstract}
 
\section{Introduction}

For $p \ge 1$ let $\ell_p$ denote the space of real-valued sequences $x \in \R^\N$ with finite $p$-th norm $\|x\|_p = (\sum_{i} |x_i|^p)^{1/p}$.
For any $n \ge 1$ and any $x_1,\ldots,x_n \in \ell_2$ there exist $y_1,\ldots, y_n \in \ell_2^n$ such that
$\|x_i - x_j \|_2 = \|y_i - y_j\|_2$ for all $i,j \in \{1,\ldots, n\}$. This is immediate from the fact that any $n$-dimensional
subspace of Hilbert space is isometric to $\ell_2^n$. In fact, there even exist such $y_1,\ldots,y_n$ in $\ell_2^{n-1}$ by considering
the $n-1$ vectors $x_2-x_1,\ldots,x_n-x_1$. We can equivalently describe this as saying that any $n$ points in $\ell_2$ can
be isometrically embedded into $\ell_2^{n-1}$. The dimension $n-1$ is easily seen to be the best possible for
isometric embeddings.

The Johnson-Lindenstrauss lemma~\cite{johnson1984extensions} establishes the striking fact that if we allow a small amount of error $\delta >0$,
a much better ``dimension reduction'' is possible. 
Namely, for any $n \ge 1$, any points $x_1,\ldots,x_n\in \ell_2$, and any $0<\delta<1$, there exist 
$n$ points $y_1,\ldots,y_n \in \ell_2^{d}$ with $d = O(\delta^{-2} \log n)$ and 
such that 
for all $i,j\in\{1,\ldots,n\}$,
\begin{equation}\label{eq:jl-intro}
 \|y_i-y_j\|_2 \,\leq\, \|x_i-x_j\|_2\,\leq\, (1+\delta) \|y_i-y_j\|_2 \;.
\end{equation}
This can be described as saying that any $n$ points in $\ell_2$ can be embedded into $\ell_2^d$ 
with (bi-Lipschitz) distortion at most $1+\delta$. 
We remark that this bound on $d$ was recently shown to be tight~\cite{LarsenN17} for essentially all values of $\delta$ for which the bound is nontrivial.

The situation for other norms is not as well understood. 
Ball~\cite{Ball90} showed that for any $p \ge 1$ and any integer $n \ge 1$, any $n$ points in $\ell_p$ embed isometrically into $\ell_p^d$ 
for $d = n(n-1)/2$. He also showed that for $1\le p < 2$ this is essentially the best possible result. 
However, if we allow some $1+\delta$ distortion as in~\eqref{eq:jl-intro}, the situation again 
changes considerably. 
Specifically, for $p=1$, Talagrand~\cite{Talagrand90} 
(improving slightly on the earlier result by Schechtman~\cite{Schechtman87})
showed that for any $0< \delta< 1$, one can embed any $n$ points in $\ell_1$ into 
$\ell_1^d$ with $d \le C \delta^{-2} n \log n$ where here and in what follows $C$ is a universal constant that might vary at each occurrence.\footnote{In fact, he showed that one can even embed
any $n$-dimensional subspace of $\ell_1$ into $\ell_1^d$ with distortion $1+\delta$.}
See also~\cite{Schechtman87,BourgainLM89,Talagrand95} for extensions to other $p$ and more details.
The bound was improved by Newman and Rabinovich~\cite{NewmanR10} to 
$d \le C n / \delta^2$ (see~\cite{Naor12}), and if we allow large enough distortion $D>1$,
the bound can be further reduced to $d \le C n / D$~\cite{AndoniNN18}.
In terms of lower bounds, Brinkman and Charikar~\cite{BrinkmanC05} showed that there exist
$n$ points in $\ell_1$ (in fact, in $\ell_1^n$) such that any embedding with distortion $D>1$ into $\ell_1^d$
requires $d \ge n^{C/D^2}$. For embeddings with distortion $1+\delta$, 
Andoni et al.~\cite{AndoniCNN11} showed a bound of $d \ge n^{1-C/\log(1/\delta)}$.
See also~\cite{LeeN04,Regev12} for alternative proofs. 

Let $\sch_1$ be the space of bounded linear operators on a separable Hilbert space with finite Schatten-1 (or nuclear) norm $\|A\|_{\schs} = \sum_{i} \sigma_i(A)$, where $\{\sigma_i(A)\}$ are the singular values of $A$. We also write $\sch_1^m$ for the space of linear operators acting on an $m$-dimensional Hilbert space, equipped with the Schatten-1 norm. 
Our main theorem shows that dimension reduction in this noncommutative analogue of $\ell_1$ is strikingly different from
that in $\ell_p$ spaces. Namely, there are $n$ points that require \emph{exponential dimension}
in any embedding with sufficiently low distortion. In contrast, Ball's result mentioned above~\cite{Ball90}
shows that in $\ell_p$, any $n$ points embed isometrically into dimension $n(n-1)/2$. 

\begin{theorem}\label{thm:realmain}
For any $n \ge 1$, there exist $(2n+2)$ points in $\sch_1^m$, where $m=2^{\lfloor n/2 \rfloor}$, such that any embedding into $\sch_1^d$ with distortion $1+\delta$ for $\delta = C n^{-c}$ requires $d \ge 2^{\lfloor n/2 \rfloor -1}$, where $c,C>0$ are universal constants. 
\end{theorem}

The space $\sch_1$ is a major object of study in many areas of mathematics and physics; 
see~\cite{NaorPS18} for further details and references. One area where it plays an especially
important role is quantum mechanics, and specifically quantum information. This area, and specifically the theory of Bell inequalities and nonlocal games, served
as an inspiration for our proof and the source of our techniques.

The best previously known bound on dimension reduction in $\sch_1$
is due to Naor, Pisier, and Schechtman~\cite{NaorPS18}, who proved
a result analogous to that of Brinkman and Charikar~\cite{BrinkmanC05}.
Namely, they showed that there exist $n$ points in $\sch_1^n$ for which any embedding
into $\sch_1^d$ with distortion $D>1$ requires $d \ge n^{C/D^2}$.\footnote{Their result is actually much stronger, and incomparable to Theorem~\ref{thm:realmain}: they show that there is no 
embedding into any $n^{C/D^2}$-dimensional subspace of $\sch_1$ (and in fact, they even allow
quotients of $\sch_1$).}
The set of points they use is \pnote{Brinkman-Charikar use the diamond graph. Naor et al. use Laakso graphs, which are very similar, and they comment that the same holds also for the diamond graphs. I therefore think it's OK not to say ``essentially''.}the one used by Brinkman and Charikar~\cite{BrinkmanC05} through the natural identification of $\ell_1^n$ with the subspace of diagonal matrices in $\sch_1^n$. The effort then goes into showing that the bound in~\cite{BrinkmanC05}, which only applies to embeddings into diagonal matrices,
also applies to arbitrary matrices. 

In Lemma~\ref{lem:ub} we show that for any $0 < \delta < 1$ the metric space induced by the $(2n+2)$ points from Theorem~\ref{thm:realmain} can be embedded with distortion $(1+\delta)$   
in $\sch_1^d$ for $d = n^{O(1/\delta^2)}$. Therefore, in order to\pnote{We can still 
hope to prove that for $\delta=0.001$, our set of points requires dimension $n^{100}$. 
Funny how in our case there really is a threshold, where either the dimension is exponential, or we don't get anything. Having something more smooth seems to require ``different techniques'' to bound the dimension of $\eps$-representations of $C(n)$ for large-ish $\eps$.}
obtain exponential lower bounds with constant $\delta$
one would have to use a different set of points.

%

\paragraph{Proof overview.} 
Due to Ball's upper bound~\cite{Ball90}, our set of points cannot be in $\ell_1$, and 
in particular, cannot be the set used in previous work~\cite{BrinkmanC05,NaorPS18}.
Instead, we introduce a new set of $n$ points in $\sch_1^m$, for $m={2^{\lfloor n/2\rfloor}}$, and show that any embedding with $(1+\delta)$ distortion for small enough $\delta$ requires almost as large a dimension. To achieve this we use metric conditions on the set of $n$ points to derive algebraic relations on any operators that (approximately) satisfy the conditions. We then conclude by applying results on the dimension of (approximate) representations of a suitable algebra. 

We now describe our construction. Let $n$ be an even integer. For a matrix $A$ and an integer $i$, let $A^{\otimes i}$ denote the tensor product of $i$ copies of $A$. Let 
\[ 
 X = \begin{pmatrix} 0 & 1 \\ 1 &0 \end{pmatrix}\; \text{,} \qquad 
 Y = \begin{pmatrix} 0 & i \\ -i &0 \end{pmatrix}\; \text{, and} \qquad 
 Z = \begin{pmatrix} 1 & 0 \\ 0 & -1 \end{pmatrix}\;.\]
For $i\in \{1,\ldots,n/2\}$ let $C_{2i-1} = X^{\otimes(i-1)} \otimes Z \otimes \Id^{\otimes(n/2-i)}$ and $C_{2i} = X^{\otimes(i-1)} \otimes Y \otimes \Id^{\otimes(n/2-i)}$. Then the matrices $C_1,\ldots,C_n$ are Hermitian operators in $\sch_1^d$, where $d=2^{n/2}$.\footnote{For a construction over the reals, consider $C'_{2i-1} = C_{2i-1}\otimes \Id$ and $C'_{2i}=C_{2i}\otimes Y$.  For even values of $n$ congruent to $4$ or $6$ mod $8$ the doubling of the dimension is necessary~\cite{okubo1991real}.} Moreover, $C_i^2 = \Id$ for each $i\in\{1,\ldots,n\}$ and $\{C_i,C_j\} = C_iC_j+C_jC_i = 0$ for $i\neq j\in\{1,\ldots,n\}$. For $i\in\{1,\ldots,n\}$ let $P_{i,+}$ (resp., $P_{i,-}$) be the projection on the $+1$ (resp., $-1$) eigenspace of $C_i$. Using that $P_{i,+}$ and $P_{i,-}$ are orthogonal  trace $0$ projectors that sum to identity, it is immediate that 
\begin{equation}\label{eq:cond-intro-1}
\forall i\in\{1,\ldots,n\}\;,\qquad \frac{1}{d}\|P_{i,+}\|_{\schs} \,=\, \frac{1}{d}\|\Id - P_{i,+}\|_{\schs} \,=\, \frac{1}{d}\|P_{i,-}\|_{\schs} \,=\, \frac{1}{d}\|\Id-P_{i,-}\|_{\schs} \,=\, \frac{1}{2}\;,
\end{equation}
 and 
\begin{equation}\label{eq:cond-intro-2}
\forall i \in \{1,\ldots,n\}\;,\qquad  \frac{1}{d}\|P_{i,+}- P_{i,-}\|_{\schs} \,=\, 1\;.
\end{equation}
Finally, using  the anti-commutation property, it follows by an easy calculation that
\begin{equation}\label{eq:cond-intro-3}
 \forall i\neq j\in\{1,\ldots,n \},\quad \forall q,r\in\{+,-\}\;,\qquad  \frac{1}{d}\|P_{i,q}-P_{j,r}\|_{\schs} \,=\, \frac{\sqrt{2}}{2}\;.
\end{equation}
 Our main result is that~\eqref{eq:cond-intro-1},~\eqref{eq:cond-intro-2} and~\eqref{eq:cond-intro-3} characterize the algebraic structure of any operators that satisfy those metric relations, even up to distortion $(1+ \delta)$ for small enough $\delta = O(n^{-c})$. Using labels $O$ and $\sigma$ to represent $0$ and $\Id/d$, and $X_i$ and $Y_i$ to represent $P_{i,+}/d$ and $P_{i,-}/d$ respectively, we show the following.  

\begin{theorem}\label{thm:main}
Let $n,d \geq 1$ be integers, $0\leq \delta \leq 1$, and $O,\sigma$ and $X_1,Y_1,\ldots,X_n,Y_n$ operators on $\C^d$ satisfying that
for all $i \in \{1,\ldots,n\}$,
\begin{align*}
1-\delta \leq \|\sigma - O\|_{\schs} &\leq 1+\delta \; ,\\
\|X_i-O\|_{\schs} + \|\sigma-X_i\|_{\schs} &\leq 1+\delta\;,\\ 
\|Y_i-O\|_{\schs} + \|\sigma-Y_i\|_{\schs} &\leq 1+\delta\;, \\ 
\|X_i-Y_i\|_{\schs} &\geq 1-\delta\;, 
\end{align*}
and for all $1 \le i < j \le n$,
\begin{align}
\min\big\{ \|X_i-X_j\|_{\schs} ,\, \|X_i-Y_j\|_{\schs} ,\, \|Y_i-X_j\|_{\schs} ,\, \|Y_i-Y_j\|_{\schs}\big\} \,\geq\, (1-\delta) \frac{\sqrt{2}}{2}\;. \label{eq:intro-ac}
\end{align}
Then there is a universal constant $C>0$ and for $i\in\{1,\ldots,n\}$ orthogonal projections $P_{i,+}$ and $P_{i,-}$ on $\C^d$ such that $P_{i,+}+P_{i,-}=\Id$ such that 
if  $A_i = P_{i,+} - P_{i,-}$ then 
\begin{equation}\label{eq:ac-intro}
\forall i\neq j \in\{1,\ldots,n\}\;,\qquad \frac{1}{d}\,\big\| A_iA_j+A_jA_i\big\|_{\scht}^2 \,\leq\, C\,n^2\, \delta^{1/32}\;.
\end{equation}
\end{theorem}

Note that the theorem does not assume that the $X_i$ and $Y_i$ are positive semidefinite, nor even that they are Hermitian; our proof shows that the metric constraints are sufficient to impose these conditions, up to a small approximation error. Similarly, while we think of $O$ as the zero matrix and of $\sigma$ as the scaled identity matrix, these conditions are not imposed a priori and have to be derived (which is very easy in the case of $O$ but less so in the case of $\sigma$). 
The proof of the theorem explicitly shows how to construct the projections $P_{i,+}$, $P_{i,-}$ from $X_i,Y_i,O$, and $\sigma$. 

Theorem~\ref{thm:realmain} follows from Theorem~\ref{thm:main} by applying known lower bounds on the dimension of (approximate) representations of the Clifford algebra that is generated by $n$ Hermitian anti-commuting operators;\footnote{Note that the norm in~\eqref{eq:ac-intro} is the Schatten-$2$ norm.} we give an essentially self-contained treatment in Section~\ref{sec:dimlowerbound}.  

The proof of Theorem~\ref{thm:main} is inspired by the theory of self-testing in quantum information theory. We interpret conditions such as~\eqref{eq:intro-ac} as requirements on the trace distance (which, up to a factor $2$ scaling, is the name used for the nuclear norm in quantum information) between post-measurement states that result from the measurement of one half of a bipartite quantum entangled state. This allows us to draw an analogy between metric conditions such as those in Theorem~\ref{thm:main} and constraints expressed by nonlocal games such as the CHSH game. Although this interpretation can serve as useful intuition for the proof, we give a self-contained proof that makes no reference to quantum information. We note that the relevance of dimension reduction for Schatten-$1$ spaces for quantum information has been recognized before; e.g., Harrow et al.~\cite{HarrowMS15} show limitations on dimension reduction maps that are restricted to be quantum channels (a result mostly superseded by~\cite{NaorPS18}).

\paragraph{Open questions.} 
We are currently not aware of \emph{any} upper bound on the dimension $d$ required to embed any $n$ points in $\sch_1$ into 
$\sch_1^d$ with, say, constant distortion. Proving such a bound would be interesting. 

Regarding possible improvements to our main theorem, our result requires the distortion of the embedding to be sufficiently small; specifically, $\delta$ needs to be at most inverse polynomial in $n$. It is open whether our result can be extended to larger distortions. 

 The connection with quantum information and nonlocal games suggests that  additional strong lower bounds may be achievable. For example, is it possible to adapt the results from~\cite{ji2018three,slofstra2018group} to construct a constant number of points in $\sch_1$ such that any embedding with distortion $(1+\delta)$ in $\sch_1^d$ requires $d \geq 2^{1/\delta^c}$ for some constant $c>0$? 

Looking at other Schatten spaces, we are only aware of trivial observations. Any set of $n$ points in $\sch_2$ trivially embeds into $\sch_2^{\lceil \sqrt{n-1}\rceil}$ by first embedding the points isometrically into $\ell_2^{n-1}$, as discussed earlier. For $\sch_\infty$, 
it is well known that any $n$ point metric isometrically embeds in $\ell_\infty^{n-1}$ and hence also in $\sch_\infty^{n-1}$; it is possible that this could be improved. We are not aware of bounds for other $\sch_p$, $p\notin\{1,2,\infty\}$.

\paragraph{Acknowledgements:}
We are grateful to IPAM and the organizers of the workshop ``Approximation Properties in Operator Algebras and Ergodic Theory'' where this work started. We also thank Assaf Naor for useful comments and encouragement. 

\section{Preliminaries}


For a matrix $A \in \C^{d\times d}$ we write $\|A\|_{\schs}$ for the Schatten-$1$ norm (the sum of the singular values). For the Schatten $2$-norm (also known as the Frobenius norm) we use $\|A\|_F$ instead of $\|A\|_{\scht}$, and introduce the dimension-normalized norm $\|A\|_f = d^{-1/2}\|A\|_F$. We write $\|A\|_{\schf}$ for the operator norm (the largest singular value). 
We  often consider terms of the form $\|T \sigma^{1/2} \|_F$ for a Hermitian matrix $T$ and a positive semidefinite matrix $\sigma$; notice that the square of this norm equals $\Tr(T^2 \sigma)$.
For $A,B$ square matrices we write $[A,B]=AB-BA$ and $\{A,B\}=AB+BA$ for the commutator and anti-commutator respectively. We write $U(d)$ for the set of unitary matrices in $\C^{d\times d}$. We use the term ``observable'' to refer to any Hermitian operator that squares to identity.


We will often use that for any $A$ and $B$,
\[ \|AB\|_{\schs} \,\leq\, \|A\|_{\schf} \|B\|_{\schs}\;,\]
and similarly with Schatten-$1$ replaced by the Frobenius norm (see, e.g.,~\cite[(IV.40)]{bhatia}). 

\begin{lemma}[Cauchy-Schwarz]\label{lem:cauchyschwarz}
For all matrices $A,B$,
\[
\|A B\|_{\schs} \,\le\, \|A\|_F \|B\|_F \; .
\]
\end{lemma}
\begin{proof}
By definition,
\[
\|A B\|_{\schs} 
\,=\, 
\sup_U \Tr(UAB)
\,\le\, 
\|U A\|_F \|B\|_F 
\,=\, 
\|A\|_F \|B\|_F \; ,
\]
where the supremum is over all unitary matrices, and the inequality follows from the Cauchy-Schwarz inequality.
\end{proof}

\section{Certifying projections}
\label{sec:projections}

In this section we prove Proposition~\ref{prop:proj}, 
showing that metric constraints on a triple of operators $(X,Y,\sigma)$, where $\sigma$ is assumed to be positive semidefinite of trace $1$, can be used to enforce that the pair $(X,Y)$ is close to a ``resolution of the identity'', in the sense that there exists a pair $(P,Q)$ of orthogonal projections such that $P+Q=\Id$ and $X\approx \sigma^{1/2}P\sigma^{1/2}$, $Y\approx\sigma^{1/2}Q\sigma^{1/2}$.
The proposition also shows that $P,Q$ approximately commute with $\sigma$. 

\begin{proposition}\label{prop:proj}
Let $\sigma$ be positive semidefinite with trace $1$. Suppose that $X$, $Y$ satisfy the following constraints, for some $0\leq \delta \leq 1$:
\begin{align}
\|X\|_{\schs} + \|\sigma-X\|_{\schs} \leq 1+\delta\;,\label{eq:norm-cond-1}\\
\quad \|Y\|_{\schs} + \|\sigma-Y\|_{\schs} \leq 1+\delta\;,\label{eq:norm-cond-2}\\
\|X-Y\|_{\schs}\geq 1-\delta\;.\label{eq:norm-cond-3}
\end{align}
Then there exist orthogonal projections $P,Q$ such that $P+Q = \Id$ and 
\begin{align}
\max\Big\{ \big\|X - \sigma^{1/2} P \sigma^{1/2} \big\|_{\schs},\, \big\|Y - \sigma^{1/2} Q \sigma^{1/2} \big\|_{\schs}\Big\}\,=\, O\big(\delta^{1/8}\big)\;,\label{eq:obs-bound}\\
\max\Big\{\big\|[P,\sigma^{1/2}] \big\|_F\;,\big\|[Q,\sigma^{1/2}] \big\|_F \Big\}\,=\,O\big(\delta^{1/8}\big)\;.\label{eq:proj-com-bound}
\end{align}
\end{proposition}

For intuition regarding  Proposition~\ref{prop:proj}, consider the case where 
$\delta=0$, and where $X,Y,\sigma$ are $1$-dimensional, i.e., scalar complex numbers, 
$X=x$, $Y=y$, and $\sigma = 1$. Then the first two conditions~\eqref{eq:norm-cond-1} and~\eqref{eq:norm-cond-2} imply that $x,y$ are real and $x,y\in[0,1]$. The third condition~\eqref{eq:norm-cond-3} then implies that $x,y\in\{0,1\}$ and $x+y=1$.
The proof of Proposition~\ref{prop:proj} follows the same outline, adapted to higher-dimensional operators. The main idea is to argue that the projections $P,Q$ on the positive and negative eigenspace of $X-Y$ respectively approximately block-diagonalize $X$, $Y$, and $\sigma$. 

The proof is broken down into a sequence of lemmas. The first lemma shows that $X$ is close to its Hermitian part. 

\begin{lemma}[Hermitianity]\label{lem:positive-herm}
Let $\sigma$ be positive semidefinite such that $\Tr(\sigma)=1$, and $X$ such that~\eqref{eq:norm-cond-1} holds, for some $0\leq \delta \leq 1$. Then $\|X-X_h\|_{\schs} \le 3\sqrt{\delta}$, where $X_h = \frac{1}{2}(X+X^*)$ is the Hermitian part of $X$. 
\end{lemma}

\begin{proof}
By~\eqref{eq:norm-cond-1}, 
\begin{align}
\Re(\Tr(X)) = 1 - \Re(\Tr(\sigma - X)) \ge 1- \|\sigma - X\|_{\schs} \ge \|X\|_{\schs} - \delta \; .\label{eq:trace-lower-bound}
\end{align}
Let $X=X_h+X_a$ be the decomposition of $X$ into Hermitian and anti-Hermitian parts. Then $\Re(\Tr(X_a))=0$, so $\Tr(X_h) \geq \|X\|_{\schs} - \delta$. Let $W$ be a unitary such that $\Tr(WX_a)=\|X_a\|_{\schs}$. Note that replacing $W \mapsto (W-W^*)/2$ we may assume that $W$ is anti-Hermitian (of norm at most $1$), so $(iW)$ is Hermitian. 
Let $0\leq \alpha \leq 1$ be a parameter to be determined. 
Then all eigenvalues of $\Id + \alpha W$ are in the complex interval $[1-\alpha i, 1+\alpha i]$ and 
therefore 
$U = (\Id + \alpha W)/(1+\alpha^2)^{1/2}$ has norm at most $1$.
Then 
\begin{align*}
\|X\|_{\schs}\,\geq\,|\Tr(UX)| &\geq \Re\big(\Tr(UX_h)+\Tr(UX_a) \big)\\ 
& = \frac{1}{(1+\alpha^2)^{1/2}} \big( \Tr(X_h) + \alpha \|X_a\|_{\schs} \big)\\
&\geq \frac{1}{(1+\alpha^2)^{1/2}} \big(\|X\|_{\schs} - \delta + \alpha \|X_a\|_{\schs}\big)\;,
\end{align*}
which shows that $\|X_a\|_{\schs} \leq \alpha\|X\|_{\schs}+ \delta/\alpha$. Choosing $\alpha = \sqrt{\delta}$ and using $\|X\|_{\schs}\leq (1+\delta)$ gives $\|X_a\|_{\schs} \le 3\sqrt{\delta}$. 
\end{proof}


\begin{lemma}\label{lem:nearlyineigenspace}
Let $X$ and $Y$ be Hermitian matrices satisfying
\begin{align*}
& \|X\|_{\schs} + \|Y\|_{\schs} \le 1+\delta \;, \\
& \|X-Y\|_{\schs} \ge 1-\delta \;,\\
& \Tr(X^-) \le \delta, \text{~and~} \Tr(Y^-) \le \delta
\end{align*}
for some $0 \le \delta \le 1$ where $X^-$ denotes the negative part of $X$ in the decomposition $X=X^+-X^-$ and similarly for $Y$. 
Then, if $P$ denotes the projection on the positive eigenspace of $X-Y$ and $Q=\Id-P$, 
we have
\[
\Tr(PX) \ge \|X\|_{\schs} - 4 \delta \; , \quad\Tr(QY) \ge \|Y\|_{\schs} - 4 \delta \; .
\]
\end{lemma}
\begin{proof}
We have
\begin{align*}
1-\delta \le \|X-Y\|_{\schs} &= \Tr(P(X-Y)) - \Tr(Q(X-Y)) \\
& \le \Tr(PX) + \Tr(QY) + 2\delta \\
& \le \Tr(PX) + \|Y\|_{\schs} + 2\delta \\
& \le \Tr(PX) + 1 + 3 \delta - \|X\|_{\schs} ,
\end{align*}
where in the second inequality we used that $\Tr(PY) \ge -\Tr(Y^-) \ge -\delta$ and similarly for $\Tr(QX)$. As a result, we get that
\[
\Tr(PX) \ge \|X\|_{\schs} - 4 \delta \; ,
\]
and similarly for $\Tr(QY)$. 
\end{proof}

\begin{lemma}\label{lem:largetraceimpliescloseness}
Let $X$ be a Hermitian matrix and $P$ a projector satisfying 
\begin{align}
& \|X\|_{\schs} \le 1, \nonumber \\
& \Tr(X^-) \le \delta, \label{eq:boundonxminusdelta}\\
& \Tr(PX) \ge \|X\|_{\schs} - \delta \label{eq:trpxupperbound}\; ,
\end{align}
for some $0 \le \delta \le 1$.
Then,
\[
\|PXP-X\|_{\schs} \le O(\sqrt{\delta})\; .
\]
\end{lemma}

\begin{proof}
The assumption~\eqref{eq:boundonxminusdelta} is equivalent to $\|X - X^+\|_{\schs} \le \delta$, which implies that
$\|PXP - PX^+P\|_{\schs} \le \delta$.
Therefore, by the triangle inequality, it suffices to prove that
\begin{align}\label{eq:goalinsubspace}
\|PX^+P-X^+\|_{\schs} \le O(\sqrt{\delta})  \; .
\end{align}
Using the Cauchy-Schwarz inequality,
\begin{align*}
\|(\Id-P)X^+\|_{\schs}^2&\leq  \|(\Id-P)(X^+)^{1/2}\|_F^2 \|(X^+)^{1/2}\|_F^2 \\
&= \Tr\big((\Id-P)X^+\big) \|X^+\|_{\schs}\\
&= \big(\Tr(X^+)- \Tr(PX) - \Tr(PX^-)\big)\|X^+\|_{\schs}\\
&\leq \delta \|X\|_{\schs} \le \delta \;,
\end{align*}
where the second line uses that $(\Id-P)$ is a projector and the fourth uses $\Tr(X^+)\leq\|X\|_{\schs}$ for the first term and~\eqref{eq:trpxupperbound} for the second. To conclude, use the triangle inequality to write
\[ \|PX^+P-X^+\|_{\schs} \,\leq\, \|(P-\Id)X^+P\|_{\schs} + \|X^+(\Id-P)\|_{\schs} \,\leq\, 2\|(\Id-P)X^+\|_{\schs}\;.\]
\end{proof}

\begin{lemma}\label{prop:nearlycommuteprojectors}
Let $\sigma$, $X$, and $Y$ satisfy the assumptions of Proposition~\ref{prop:proj}
for some $0\leq \delta \leq 1$.
Then there exist orthogonal projections $P,Q$ such that $P+Q = \Id$ and 
\begin{align}
\|X - P\sigma P\|_{\schs} \le O(\delta^{1/4}) \text{~and~} \|Y - Q\sigma Q\|_{\schs} \le O(\delta^{1/4}).
\label{eq:prop:nearlycommuteprojectors}
\end{align}
Moreover, there exists a positive semidefinite $\rho$ that commutes with $P$ and $Q$ and
that satisfies $\|\rho-\sigma\|_{\schs} \le O(\delta^{1/4})$.
\end{lemma}
\begin{proof}
Using Lemma~\ref{lem:positive-herm}, we can replace $X$ and $Y$ with their Hermitian parts,
and have Eqs.~\eqref{eq:norm-cond-1}-\eqref{eq:norm-cond-3} still hold with $O(\sqrt{\delta})$ 
in place of $\delta$.
By summing Eqs.~\eqref{eq:norm-cond-1} and~\eqref{eq:norm-cond-2}, and noting by the triangle inequality
that $\|\sigma-X\|_{\schs} + \|\sigma-Y\|_{\schs} \ge \|X-Y\|_{\schs} \ge 1-O(\sqrt \delta)$, we get that
$\|X\|_{\schs} + \|Y\|_{\schs} \le 1+O(\sqrt \delta)$. Moreover,
\begin{align*}
\Tr(X^-) &= \|X^+\|_{\schs} - \Tr(X) \\
&\le 1+O(\sqrt\delta) - \|\sigma-X\|_{\schs} - \Tr(X) \\
& \le 1+O(\sqrt\delta) - \Tr(\sigma-X) - \Tr(X) = O(\sqrt\delta) 
\end{align*}
and similarly for $Y$. 
We can therefore apply Lemma~\ref{lem:nearlyineigenspace} and obtain that if 
$P$ is the projection on the positive eigenspace of $X-Y$ and $Q=\Id-P$,
\[
\Tr(PX) \ge \|X\|_{\schs} - O(\sqrt{\delta})  \text{~and~}
\Tr(QY) \ge \|Y\|_{\schs} - O(\sqrt{\delta}) \; .
\]
Applying Lemma~\ref{lem:largetraceimpliescloseness} to $X$ (scaled by a factor at most $(1+\delta)$ so that
the condition $\|X\|_{\schs}\leq 1$ is satisfied) and $P$, we get that
\begin{align}\label{eq:XYinsubspace}
\|PXP - X\|_{\schs} = O(\delta^{1/4})  \text{~and~}
\|QYQ - Y\|_{\schs} = O(\delta^{1/4})  \; .
\end{align}
Notice that the set of constraints in Eqs.~\eqref{eq:norm-cond-1}-\eqref{eq:norm-cond-3} is invariant under
replacing the pair $(X,Y)$ with $(\sigma-Y,\sigma-X)$. Moreover, our assumption that
$X$ and $Y$ are Hermitian implies that $\sigma-X$ and $\sigma-Y$ are also Hermitian.
Therefore, the exact same argument as above applies also to $\sigma-X$ and $\sigma-Y$ 
and we conclude that
\begin{align}\label{eq:sigmaXYinsubspace}
\|P(\sigma-Y)P - (\sigma-Y)\|_{\schs} = O(\delta^{1/4})
  \text{~and~}
	\|Q(\sigma-X)Q - (\sigma-X)\|_{\schs} = O(\delta^{1/4})
  \; .
\end{align}
Notice that we used here the fact that $(\sigma-Y)-(\sigma-X) = X-Y$ and therefore
the projections $P$ and $Q$ obtained when we apply Lemma~\ref{lem:nearlyineigenspace}
to $X$ and $Y$ are identical to those obtained when we apply it to
$\sigma-Y$ and $\sigma-X$.

From~\eqref{eq:sigmaXYinsubspace}, and since $PQ=0$, we obtain that
\[
\|P\sigma P - PXP\|_{\schs} =
\|PQ(\sigma-X)QP - P(\sigma-X)P\|_{\schs} \le
\|Q(\sigma-X)Q - (\sigma-X)\|_{\schs} = O(\delta^{1/4}) \; .
\]
Together with~\eqref{eq:XYinsubspace} and the triangle inequality, this 
proves~\eqref{eq:prop:nearlycommuteprojectors}. 

To prove the last part
of the lemma, let $\tilde\rho = PXP + Q(\sigma-X)Q$
and notice that $\tilde{\rho}$ commutes with $P$ and $Q$. By Eqs.~\eqref{eq:XYinsubspace} 
and~\eqref{eq:sigmaXYinsubspace} and the triangle inequality,
$\|\tilde\rho - \sigma \|_{\schs} = O(\delta^{1/4})$. Finally, we define $\rho$ to be 
the positive part of $\tilde\rho$, which due to the block diagonal form of $\tilde\rho$ still commutes with $P$ and $Q$. We have
$\|\rho - \sigma \|_{\schs} = O(\delta^{1/4})$ since
\[
\|\rho - \tilde\rho \|_{\schs} = \frac{1}{2} ( \|\tilde \rho\|_{\schs} - \Tr(\tilde \rho) ) 
\le \frac{1}{2} (\|\sigma\|_{\schs} - \Tr(\sigma)) + O(\delta^{1/4})
= 
O(\delta^{1/4}) \; ,
\]
where the last equality uses that $\sigma$ is positive semidefinite.
\end{proof}

We conclude by giving the proof of Proposition~\ref{prop:proj}.

\begin{proof}[Proof of Proposition~\ref{prop:proj}]
Let $P$, $Q$, and $\rho$ be as guaranteed by Lemma~\ref{prop:nearlycommuteprojectors}.
Using the Powers-Stormer inequality $\|\sqrt{R}-\sqrt{S}\|_F \leq \|R-S\|_{\schs}^{1/2}$ for positive semidefinite $R$, $S$ (see, e.g.,~\cite[(X.7)]{bhatia}), it follows that
\begin{align}\label{eq:powersstormer}
\| \rho^{1/2}-\sigma^{1/2} \|_F \le \| \rho - \sigma \|_{\schs}^{1/2} = O(\delta^{1/8}) \; .
\end{align}
As a result, using the triangle inequality and Cauchy-Schwarz,
\begin{align*}
\| \sigma^{1/2} P \sigma^{1/2} - \rho^{1/2} P \rho^{1/2} \|_{\schs} \le
\| (\sigma^{1/2} - \rho^{1/2}) P \sigma^{1/2} \|_{\schs} +
\| \rho^{1/2} P (\sigma^{1/2} - \rho^{1/2}) \|_{\schs} \le O(\delta^{1/8}) \;,
\end{align*}
where we used that $\|P \sigma^{1/2}\|_F \le \| \sigma^{1/2} \|_F = 1$
and $\|P \rho^{1/2}\|_F \le \| \rho^{1/2} \|_F = 1 + O(\delta^{1/4})$.
But $\rho$ commutes with $P$ and therefore
$\rho^{1/2} P \rho^{1/2} = P \rho P$, and we complete the proof of~\eqref{eq:obs-bound} by noting
that 
\[
\|P \rho P - P \sigma P\|_{\schs} \le \|\rho-\sigma\|_{\schs} = O(\delta^{1/4}) \; .
\]

To prove~\eqref{eq:proj-com-bound}, notice that by~\eqref{eq:powersstormer} and the triangle inequality,
\[
\| P \sigma^{1/2} - \sigma^{1/2} P \|_F \le 
\| P \rho^{1/2} - \rho^{1/2} P \|_F + O(\delta^{1/8}) \;,
\]
but the latter norm is zero since $P$ commutes with $\rho$.
\end{proof}


\section{Certifying anticommutation}
\label{sec:ac}

In this section we prove Proposition~\ref{prop:ac}. The proposition shows that assuming two pairs of operators $(X_1,Y_1)$ and $(X_2,Y_2)$ satisfying the assumptions of Proposition~\ref{prop:proj} satisfy additional metric constraints, the corresponding projections $(P_1,Q_1)$ and $(P_2,Q_2)$ are such that the operators $P_1-Q_1$ and $P_2-Q_2$ have small anti-commutator, in the appropriate norm. For intuition, consider the case of operators in two dimensions, and $\sigma=\Id$. 
Then, Proposition~\ref{prop:proj} shows that we can think of $(X_1,Y_1)$ and $(X_2,Y_2)$ as two pairs of orthogonal projections.
Assuming that these projections are of rank $1$ (as would follow from the constraint~\eqref{eq:antic-cond} below),
we can think of them as two pairs of orthonormal bases $(u_1,v_1)$ and $(u_2,v_2)$ of $\C^2$. Suppose we were to impose that these vectors satisfy the four Euclidean conditions 
\begin{align}
  \|u_1-u_2\|_2^2 \,=\, \|u_1-v_2\|_2^2 \,=\, \|v_1-u_2\|_2^2 \,=\, \|v_1+v_2\|_2^2 \,=\, 2-\sqrt{2}\;.
	\label{eq:twodimensionalexample}
\end{align}
By expanding the squares, it is not hard to see that these conditions imply that the bases must form an angle of $\frac{\pi}{4}$ as shown in Figure~\ref{fig:chsh}.\footnote{These conditions underlie the rigid properties of the famous CHSH inequality from quantum information~\cite{tsirelson,summers1987maximal}.}
In particular, the reflection operators $A_i = u_i u_i^* - v_i v_i^*$, $i\in\{1,2\}$, anti-commute. Proposition~\ref{prop:ac} adapts this observation to the trace norm between matrices in any dimension, and small error.
We start with two technical claims. 

\begin{figure}[ht]
\centering
\includegraphics[width=30mm]{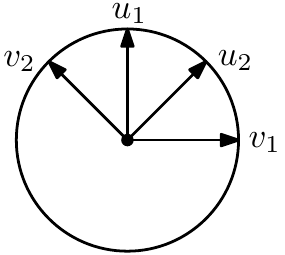}
\caption{Vectors satisfying the metric constraints}
\label{fig:chsh}
\end{figure}

\begin{claim}\label{claim:c-s-equality}
Let $A,B\neq 0$ be such that $\Re(\Tr(A^* B))\geq (1-\delta)\|A\|_F\|B\|_F$ for some $0\leq\delta \leq 1$. Let $\alpha = \|A\|_F/\|B\|_F$. Then $\|A-\alpha B\|_F \leq \sqrt{2\delta}\|A\|_F$.  
\end{claim}

\begin{proof}
Expand 
\begin{align*}
\|A-\alpha B\|_F^2\, 
&=\, \|A\|_F^2 + \alpha^2 \|B\|_F^2 - 2\alpha\Re(\Tr (A^* B)) \\
& \le \,  \|A\|_F^2 + \alpha^2 \|B\|_F^2 - 2\alpha (1-\delta) \|A\|_F \|B\|_F \\
& = \, 2 \delta \|A\|_F^2 
\;.
\end{align*}
\end{proof}

\begin{claim}\label{claim:cs-cor}
Let $R$ be Hermitian and $\sigma$ positive semidefinite such that $\Tr(\sigma)=1$. Suppose further that $\|\sigma^{1/2}R\sigma^{1/2}\|_{\schs} \geq (1-\delta) \sqrt{\mu}$, where $\mu = \Tr(R^2\sigma)$. Then 
\[ \big\|(R^2-\mu\Id)\sigma^{1/2}\big\|_F^2\,=\, O\big(\sqrt{\delta}\|R\|_{\schf}^2\big) \,\mu\;.\]
\end{claim}

\begin{proof}
Let $U$ be a unitary such that $U\sigma^{1/2}R\sigma^{1/2} =|\sigma^{1/2}R\sigma^{1/2}|$. Let $A = R\sigma^{1/2}$ and $B = \sigma^{1/2}U$,
and notice that $\|A\|_F = \sqrt \mu$ and $\|B\|_F = 1$. Then 
\[ 
\Tr\big(A^* B\big)
\,=\, 
\Tr\big(\sigma^{1/2} R \sigma^{1/2}U \big) 
\,=\,
\Tr|\sigma^{1/2} R\sigma^{1/2}|\,\geq\, (1-\delta)\sqrt{\mu}\;,
\]
by assumption. Applying Claim~\ref{claim:c-s-equality}  it follows that 
\begin{equation}\label{eq:cs-cor-1}
\|R\sigma^{1/2} - \sqrt{\mu} \sigma^{1/2} U\|_F^2\leq 2\delta \mu\;.
\end{equation}
By the triangle inequality,
\begin{align*}
\|R\sigma R - \mu \sigma\|_{\schs} &\leq  \| (R\sigma^{1/2}- \sqrt{\mu} \sigma^{1/2} U)\sigma^{1/2}R\|_{\schs} +  \|\sqrt{\mu} \sigma^{1/2} U (\sqrt{\mu} U^* \sigma^{1/2}-  \sigma^{1/2} R)\|_{\schs} \\
&\leq 2\sqrt{2\delta}\mu\;,
\end{align*}
where the second line uses the Cauchy-Schwarz inequality and~\eqref{eq:cs-cor-1}. 
Thus
\begin{align*}
\Tr\big( (R^2-\mu\Id)^2\sigma\big) &= \Tr\big(R^4\sigma\big) - 2\mu\Tr\big(R^2\sigma\big) + \mu^2 \\
&=\Tr\big(R^2(R\sigma R - \mu \sigma)\big)  \\
&\leq  2\sqrt{2\delta}\|R\|_{\schf}^2  \mu\;.
\end{align*}
\end{proof}

\begin{proposition}\label{prop:ac}
Let $\sigma$ be positive semidefinite such that $\Tr(\sigma)=1$. Let $X_1,Y_1$ and $X_2,Y_2$ be operators satisfying the assumptions of Proposition~\ref{prop:proj} for some $0\leq \delta \le 1$, and $P_1,Q_1$ and $P_2,Q_2$ be as in the conclusion of the proposition.  Suppose further that\footnote{The reason that the ``$+$'' sign in the last term in~\eqref{eq:twodimensionalexample} is replaced by a ``$-$'' in~\eqref{eq:antic-cond} is that one should think of $X_i,Y_j$ as the projectors on $u_i,v_j$.}
\begin{align}
\min\big\{ \|X_1-X_2\|_{\schs} ,\, \|X_1-Y_2\|_{\schs} ,\, \|Y_1-X_2\|_{\schs} ,\, \|Y_1-Y_2\|_{\schs}\big\} \,\geq\, (1-\delta) \frac{\sqrt{2}}{2}\;.\label{eq:antic-cond}
\end{align}
For $i\in\{1,2\}$ let $A_i = P_i-Q_i$. Then $A_1$, $A_2$ are observables\footnote{Recall that an observable is a Hermitian operator that squares to identity.} such that 
\begin{align}
 \big\| \{A_1,A_2\} \sigma^{1/2} \big\|_F \,=\, O\big(\delta^{1/32}\big)\;.\label{eq:obs-ac-bound}
\end{align}
\end{proposition}

\begin{proof}
Using first~\eqref{eq:obs-bound} and then the Cauchy-Schwarz inequality and $\Tr(\sigma) = 1$,
 \begin{align}
\frac{1-2\delta}{2} \le \|X_1-X_2\|_{\schs}^2 &\leq \big\|\sigma^{1/2} (P_1-P_2)\sigma^{1/2}\big\|_{\schs}^2 + O\big(\delta^{1/8}\big)\notag\\
&\leq \Tr\big( (P_1-P_2)^2 \sigma\big)+ O\big(\delta^{1/8}\big) \label{eq:normapprox-2}
\end{align}
and similarly for the three other pairs ($X_1-Y_2$, $Y_1-X_2$, and $Y_1-Y_2$). 
Summing those four inequalities, we get
\begin{align*}
2(1-2\delta) &\le 
\Tr\big( (P_1-P_2)^2 \sigma\big) + \Tr\big( (P_1-Q_2)^2 \sigma\big)
+\Tr\big( (Q_1-P_2)^2 \sigma\big) + \Tr\big( (Q_1-Q_2)^2 \sigma\big)+ O\big(\delta^{1/8}\big)\\
&=
2 \big( \Tr\big( (P_1-P_2)^2 \sigma\big) + \Tr\big( (P_1+P_2-\Id)^2 \sigma\big) \big)
+ O\big(\delta^{1/8}\big)\\
&= 2 + O\big(\delta^{1/8}\big)\;,
\end{align*}
where the first equality uses $Q_1 -Q_2 = P_2-P_1$ and $Q_1-P_2 = Q_2 - P_1$, and the second uses $\Tr(\sigma)=1$.
Therefore all inequalities in~\eqref{eq:normapprox-2} must be equalities, up to $O(\delta^{1/8})$. 
Applying Claim~\ref{claim:cs-cor} to the tightness of~\eqref{eq:normapprox-2}, it follows that 
\begin{equation}\label{eq:ac-1-0}
 \big\|\big((P_1-P_2)^2 - \frac{1}{2} \Id\big)\sigma^{1/2} \big\|_F^2 \,=\, O\big(\delta^{1/16}\big)\;,
\end{equation}
and similar bounds for the three other pairs. To conclude the proof, use the triangle inequality, Eq.~\eqref{eq:ac-1-0}, and the observation that by writing $A_1=2P_1-\Id$ and $A_2=2P_2-\Id$, 
\begin{align*}
\big\{A_1,A_2 \big\} &\,=\, 
4 P_1 P_2 + 4P_2 P_1 - 4P_1 - 4P_2 + 2 \Id \\
&\, = \,
2\big((P_1-Q_2)^2 - (P_1-P_2)^2\big)\;.
\end{align*}
%
\end{proof}

\section{Replacing \texorpdfstring{$\sigma$}{sigma} with identity}

The anti-commutation relations obtained in Proposition~\ref{prop:ac} involve the arbitrary positive semidefinite operator $\sigma$. In this section we show that up to a small loss of parameters we may without loss of generality assume that $\sigma = \Id$. Intuitively, this follows from the approximate commutation relation
\begin{align}
\big\|[A,\sigma^{1/2}] \big\|_F \,=\,O\big(\delta^{1/8}\big)\;,\label{eq:obs-com-bound}
\end{align}
which follows immediately from the definition of the observable $A = P-Q$ and~\eqref{eq:proj-com-bound}. 
If $\sigma$ has two eigenvalues with a big gap between them, then it is not hard to see that $A$ satisfying~\eqref{eq:obs-com-bound} must have a corresponding approximate block structure, in which case we can restrict to one of the blocks
and obtain $\sigma=\Id$ as desired.  
The difficulty is in carefully handling the general case, where some eigenvalues of $\sigma$ might be closely spaced.
The following lemma does this, using an elegant argument borrowed from~\cite{Slofstra2018}.
\begin{lemma}\label{lem:connes}
Let $\sigma$ be a positive semidefinite matrix with trace $1$, and $T_1,\ldots,T_k$ and $X_1,\ldots,X_\ell$ Hermitian operators such that $X_j^2=\Id$ for all $j\in\{1,\ldots,\ell\}$. Let
\[\eps \,= \, \frac{1}{k}\sum_{i=1}^k \big\|T_i\sigma^{1/2}\big\|_F^2\qquad\text{ and }\qquad\delta \,=\, \frac{1}{\ell}\sum_{j=1}^\ell \big\|[X_j,\sigma^{1/2}]\big\|_F^2\;.\]
Then there exists a nonzero orthogonal projection $R$ such that 
\[ \frac{1}{k} \sum_{i=1}^k \big\|T_i R\big\|_F^2\,=\, O(\eps)\Tr(R)\qquad\text{ and }\qquad  \frac{1}{\ell}\sum_{j=1}^\ell \big\|[X_j,R]\big\|_F^2 \,=\,O\big(\delta^{1/2}\big)\Tr(R)\;.\]
\end{lemma}

\begin{proof}
The proof relies on two simple claims. 
For a Hermitian matrix $\rho$ and $a \ge 0$, let $\chi_{\geq a}(\rho)$ denote the projection on the direct sum of eigenspaces of $\rho$
with eigenvalues at least $a$. 
The first claim appears as~\cite[Lemma 5.6]{Slofstra2018}. 

\begin{claim}\label{L:int}
    Let $\rho$ be positive semidefinite. Then
    \begin{equation*}
        \int_{0}^{+\infty} \chi_{\geq \sqrt{a}}(\rho)\, da \,=\, \rho^2\;.
    \end{equation*}
\end{claim}

The second is due to Connes~\cite[Lemma 1.2.6]{Connes76}. We state the claim as it appears in~\cite[Lemma 5.5]{Slofstra2018}. 

\begin{claim}[\cite{Connes76}, Lemma 1.2.6]\label{L:connestrick}
    Let $\rho, \rho'$ be positive semidefinite. Then 
    \begin{equation*}
        \int_0^{+\infty} \big\|\chi_{\geq \sqrt{a}}(\rho) 
            - \chi_{\geq \sqrt{a}}(\rho')\big\|^2_F\,da 
            \,\leq \, \|\rho - \rho'\|_F \|\rho + \rho'\|_F\;.
    \end{equation*}
\end{claim}

Both claims can be proven by direct calculation, writing out the spectral decomposition of $\rho,\rho'$ and using Fubini's theorem (exchanging summation indices). The proof is given in~\cite{Slofstra2018}.

Applying Claim~\ref{L:int} with $\rho = \sigma^{1/2}$,
\begin{equation}\label{eq:connes-1}
 \frac{1}{k} \int_0^{+\infty} \sum_{i=1}^k \big\|T_i\,\chi_{\geq \sqrt{a}}\big(\sigma^{1/2}\big)\big\|_F^2\, da 
\,=\, 
\frac{1}{k} \sum_{i=1}^k \|T_i \sigma^{1/2}\big\|_F^2
\,\leq\,
\eps\int_0^{+\infty} \Tr\big(\chi_{\geq \sqrt{a}}\big(\sigma^{1/2}\big)\big)\,da\;,
\end{equation}
where the 
first equality uses 
$\|T_i\,\chi_{\geq \sqrt{a}}(\sigma^{1/2})\|_F^2 = 
\Tr(T_i^2 \,\chi_{\geq \sqrt{a}}(\sigma^{1/2}))$
and the second inequality follows from Claim~\ref{L:int} and $\Tr(\sigma)=1$.
Applying Claim~\ref{L:connestrick} with $\rho = \sigma^{1/2}$ and $\rho' =  X_j \sigma^{1/2}X_j$, and using that $X_j$ is Hermitian and unitary, 
\begin{align}
\frac{1}{\ell}  \int_0^{+\infty} \sum_{j=1}^\ell \big\|\big[X_j,\chi_{\geq \sqrt{a}}\big(\sigma^{1/2}\big)\big]\big\|_F^2 \,da 
&\leq 
\frac{1}{\ell}  \sum_{j=1}^\ell \|[X_j, \sigma^{1/2}]\big\|_F\|\{X_j, \sigma^{1/2}\}\big\|_F\notag\\
&\leq O\big(\delta^{1/2}\big)\int_0^{+\infty} \Tr\big(\chi_{\geq \sqrt{a}}\big(\sigma^{1/2}\big)\big)\,da\;,\label{eq:connes-2}
\end{align}
where the second inequality follows from the Cauchy-Schwarz inequality and uses $\Tr(\sigma)=1$ and $\|X_j\|_{\schf}\leq 1$. Adding $(1/\eps$) times~\eqref{eq:connes-1} and $(1/\delta^{1/2})$ times~\eqref{eq:connes-2}, there exists an $a\geq 0$ such that both inequalities are satisfied simultaneously (up to a multiplicative constant factor loss) with a nonzero right-hand side, for that $a$. 
Then $R = \chi_{\geq \sqrt{a}}(\sigma^{1/2})$ is a projection that satisfies the conclusions of the lemma. 
\end{proof}

Combining Proposition~\ref{prop:ac} and Lemma~\ref{lem:connes}, we obtain the following.

\begin{proposition} \label{prop:me}
Let $n,d \geq 1$ be integers, $0\leq \delta \leq 1$, $X_1,Y_1,\ldots,X_n,Y_n$ operators on $\C^d$, and $\sigma$ positive semidefinite of trace $1$, such that for each $i\in\{1,\ldots,n\}$, $\sigma,X_i,Y_i$ satisfy~\eqref{eq:norm-cond-1},~\eqref{eq:norm-cond-2},~\eqref{eq:norm-cond-3}, and such that for each $i\neq j\in\{1,\ldots,n\}$, $(X_i,Y_i,X_j,Y_j)$ satisfy~\eqref{eq:antic-cond}.
Then there exist a $d'\leq d$ and observables $A'_1,\ldots,A'_n$ on $\C^{d'}$ such that 
\[ 
\frac{2}{n(n-1)} \sum_{1 \le i < j \le n}\, \big\| \{A'_i,A'_j\}  \big\|_f^2 \,=\, O\big(\delta^{1/32}\big)\;.
\]
\end{proposition}

\begin{proof}
Applying Proposition~\ref{prop:ac} and~\eqref{eq:obs-com-bound} we deduce the existence of observables $A_1,\ldots,A_n$ on $\C^d$ such that
\begin{align}
\forall i\neq j\in\{1,\ldots,n\}\;,\quad \big\| \{A_i,A_j\} \sigma^{1/2} \big\|_F^2 \,=\, O\big(\delta^{1/32}\big)\;,\\
\forall i\in\{1,\ldots,n\}\;,\quad\big\|[A_i,\sigma^{1/2}] \big\|_F^2\,=\,O\big(\delta^{1/4}\big)\;.
\end{align}
(Note that this uses that for each $i\in\{1,\ldots,n\}$, the projections $P_i,Q_i$ used to define $A_i=P_i-Q_i$ depend on $X_i$ and $Y_i$ only.)  
Next apply Lemma~\ref{lem:connes} with $T_{ij} =  \{A_i,A_j\}$ and $X_i=A_i$. The lemma gives
an orthogonal projection $R$ on $\C^d$ such that 
\begin{align}
\frac{2}{n(n-1)} \sum_{ 1 \le i < j \le n}  \big\| \{A_i,A_j\} R\big\|_F^2 \,=\, O\big(\delta^{1/32}\big)\Tr(R)\;,\label{eq:me-1}\\
\frac{1}{n}\sum_{i\in\{1,\ldots,n\}} \quad\big\|[A_i,R] \big\|_F^2\,=\,O\big(\delta^{1/8}\big)\Tr(R)\;.\label{eq:me-2}
\end{align}
For $i\in\{1,\ldots,n\}$ let $\tilde{A}_i = RA_iR$. Then using $A_i^2=\Id$ and $R^2=R$,
\begin{align*}
\big\| \tilde{A}_i^2-R\big\|_F &= \big\| R[A_i,R]A_iR\|_F\\
&\leq \|[A_i,R]\|_F \;.
\end{align*}
Defining the observable $A'_i = R\sign(\tilde{A}_i)R$ and using the inequality $(\sign(x)-x)^2 \le (x^2-1)^2$ valid for all $x \in [-1,1]$, 
we see that $\|A'_i - \tilde{A}_i\|_F \leq \| \tilde{A}_i^2-R\|_F$. 
Then using~\eqref{eq:me-2},
\begin{equation}\label{eq:me-3}
\frac{1}{n}\sum_{i\in\{1,\ldots,n\}} \quad\big\|A'_i - \tilde{A}_i \big\|_F^2 \,=\,O\big(\delta^{1/8}\big)\Tr(R)\;.
\end{equation}
For any $i,j$, using the triangle inequality
\begin{align*}
\big\| \{A'_i,A'_j\}  \big\|_F 
&\leq 
\big\| \{\tilde{A}_i,\tilde{A}_j\}  \big\|_F +  2\big(\big\| A'_i - \tilde{A}_i  \big\|_F + \big\| A'_j - \tilde{A}_j  \big\|_F\big) \\
&\leq 
\big\| \{A_i,A_j\} R \big\|_F +  2\big(\|[A_i,R]\|_F + \|[A_j,R]\|_F\big) + 2\big(\big\| A'_i - \tilde{A}_i  \big\|_F + \big\| A'_j - \tilde{A}_j  \big\|_F\big)\;,
\end{align*}
where the second inequality uses the definition of $\tilde{A}_i$ and $\tilde{A}_j$. 
Squaring this inequality and using Cauchy-Schwarz gives
\begin{align*}
\big\| \{A'_i,A'_j\}  \big\|_F^2 
&\leq 
O\Big( \big\| \{A_i,A_j\} R \big\|_F^2 +  \|[A_i,R]\|_F^2 + \|[A_j,R]\|_F^2 + \big\| A'_i - \tilde{A}_i  \big\|_F^2 + \big\| A'_j - \tilde{A}_j  \big\|_F^2 \Big)\;,
\end{align*}
Averaging over all pairs $i\neq j$ and using~\eqref{eq:me-1} and~\eqref{eq:me-3} proves the proposition. 
\end{proof}

\begin{proof}[Proof of Theorem~\ref{thm:main}]
By subtracting $O$ from all the operators, we can assume without loss of generality that $O$ is zero.
Let $U$ be a unitary such that $\sigma=U|\sigma|$, as given by the polar decomposition. Multiplying all operators on the left by $U^{-1}$, we may further assume that $\sigma$ is positive semidefinite. Dividing by $\|\sigma\|_{\schs}$, we may assume that $\Tr(\sigma)=1$, and $\delta$ is replaced by $\delta'=O(\delta)$. 
Eq.~\eqref{eq:ac-intro} now follows from Proposition~\ref{prop:me}.
\end{proof}

\section{Dimension bounds}\label{sec:dimlowerbound}

The following lemma shows that pairwise approximately anti-commuting observables only exist in large dimension. The observation is not new; see, e.g.,~\cite{ostrev2016entanglement,slofstra2018group}. We give a proof that closely follows~\cite{slofstra2018group}.
Theorem~\ref{thm:realmain} follows immediately by combining the lemma with Theorem~\ref{thm:main}, 
provided $\delta^{1/64} \leq C/n^{3}$ for some sufficiently small constant $C$.

\begin{lemma}\label{lem:clifford}
Let $n\geq 2$ and $d \ge 1$ be integers, $0\leq \eps \leq 1$, and $A_1,\ldots,A_n$ observables on $\C^d$ such that
\begin{equation}\label{eq:clifford-cond}
 \forall i\neq j\in\{1,\ldots,n\}\;,\quad \big\| \{A_i,A_j\}  \big\|_f\,\leq\,\eps\;.
\end{equation}
Then there are universal constants $c,C>0$ such that if $n^2\eps \leq c$ then $d\geq (1-C n^4 \eps^2) 2^{\lfloor n/2\rfloor}$.
\end{lemma}

\begin{proof}
The idea for the proof is that if $\eps=0$, then the $A_i$ would induce a representation of the (finite) finitely presented group
\[C(n)=\big\langle J,x_1,\ldots,x_n:\, Jx_i=x_iJ, J^2=x_i^2=1,x_ix_j=Jx_jx_i \,\text{ for all }i\neq j\in\{1,\ldots,n\}\big\rangle\;\]
such that moreover, the representation maps $J$ to $-\Id$. 
Depending on the parity of $n$, the group $C(n)$ has either one or two irreducible representations such that $J\mapsto -\Id$, each of dimension $2^{\lfloor n/2\rfloor}$, implying a corresponding lower bound on the dimension $d$ of the $A_i$. 
The goal for the proof is to extend this lower bound to $\eps >0$. This is done in ~\cite{slofstra2018group} (see Lemma 3.1 and Lemma 3.4). There are two steps: first, we use $A_i$ satisfying~\eqref{eq:clifford-cond} to define an approximate homomorphism on $C(n)$ such that $J\mapsto -\Id$. 
Second, we use a stability theorem due to Gowers and Hatami~\cite{gowers2015inverse} to argue that any such approximate homomorphism is close to an exact one, and hence must have large dimension. 

The first step is given by the following claim, a slightly simplified version of~\cite[Lemma 3.4]{slofstra2018group}. 

\begin{claim}[Lemma 3.4 in \cite{slofstra2018group}]\label{lem:s1}
Let $A_1,\ldots,A_n$ satisfy the conditions of Lemma~\ref{lem:clifford}. For any $x = J^a x_{i_1}\cdots x_{i_k}$, where $1\leq i_1<\cdots<i_k\leq n$, define $\phi(x)=(-1)^a A_{i_1}\cdots A_{i_n}$. Then  $\phi$ is an $\eta = n^2{\eps}$-homomorphism from $C(n)$ to $U(d)$, i.e., for every $x,y\in C(n)$ it holds that $\|\phi(xy)-\phi(x)\phi(y)\|_f \leq \eta$.
\end{claim}

\begin{proof}
Any element of $C(n)$ has a unique representation of the form described in the claim. Let $x,y\in C(n)$ such that $x = J^a x_{i_1}\cdots x_{i_k}$ and $y = J^b x_{j_1}\cdots x_{j_\ell}$. To write $xy$ in canonical form involves at most $n^2$ application of the anti-commutation relations to sort the $\{x_i,x_j\}$ (together with a number of commutations of $J$ with the $x_i$, that we need not count since in our representation $\phi(J)=-\Id$ commutes with all $A_i$), and finally at most $n$ application of the relations $x_i^2=1$. When considering $\phi(x)$ and $\phi(y)$, the only operation that is not exact is the anti-commutation between different $A_i,A_j$. Using the triangle inequality, $\|\phi(xy)-\phi(x)\phi(y)\|_f \leq n^2 {\eps}$, as desired. 
\end{proof}

The second step of the proof is given by the following lemma from~\cite{slofstra2018group}, which builds on~\cite{gowers2015inverse}. 

\begin{lemma}[Lemma 3.1 in \cite{slofstra2018group}]\label{lem:s2}
Let $\phi$ be a map from $C(n)$ to the set of unitaries in $d$ dimensions such that $\phi$ is an $\eta$-homomorphism for some $0\leq \eta\leq 1$. Suppose furthermore that $\|\phi(J)-\Id\|_f > 42 \eta$. Then $d \geq (1-4\eta^2)2^{\lfloor n/2\rfloor}$. 
\end{lemma}

The proof of the lemma first applies the results from~\cite{gowers2015inverse} to argue that $\phi$ must be close to an exact representation of $C(n)$, and then concludes using that all irreducible representations of $C(n)$ that send $J$ to $(-\Id)$ have dimension $2^{\lfloor n/2\rfloor}$. 

Combining Claim~\ref{lem:s1} and Lemma~\ref{lem:s2} proves Lemma~\ref{lem:clifford}. 
\end{proof}

We conclude this section by a construction showing that the metric space implied by the $(2n+2)$ points from Theorem~\ref{thm:realmain} can be embedded with constant distortion in a Schatten-$1$ space of polynomial dimension. The construction is inspired by a result of Tsirelson~\cite{tsirelson} in quantum information. 

\begin{lemma}\label{lem:ub}
Let $n\geq 1$ be an integer and $0<\delta< 1$.
There exists a $(1+\delta)$ distortion embedding of the metric space induced by the $(n+2)$ points from Theorem~\ref{thm:realmain} into $\sch_1^d$ with $d = n^{O(1/\delta^2)}$.
\end{lemma}

\begin{proof}
For simplicity, assume that $n$ is even. 
To show the lemma we construct real operators $O=0$, $\sigma = \frac{1}{d}\Id$, and $X_1,\ldots,X_n$ and $Y_1,\ldots,Y_n$ in $\sch_1^{d}$ that approximately satisfy the metric relations implied by the $(2n+2)$ points from Theorem~\ref{thm:main}, i.e., the operators $0$, $2^{-\frac{n}{2}} \Id$, and $2^{-\frac{n}{2}}P_{i,+}$, $2^{-\frac{n}{2}}P_{i,-}\in\sch_1^{2^{n/2}}$ defined in the introduction. 

By the Johnson-Lindenstrauss lemma~\cite{johnson1984extensions} there are $n$ unit vectors $x_1,\ldots,x_n \in \R^d$ for $d\leq C\ln n/\delta^2$ such that the inner products $|x_i\cdot x_j|\leq \delta/4$ for all $i\neq j$. Let $C_1,\ldots,C_d$ be a real representation of the Clifford algebra, i.e., real symmetric matrices such that $\{C_i,C_j\}=C_iC_j+C_jC_i=2\delta_{ij}\Id$ for all $i,j$, where $\delta_{ij}$ is the Kronecker coefficient. As already mentioned in the introduction, there always exists such a representation of dimension $2^{d'}$ for $d'\leq \lceil d/2\rceil +1$. For $i\in\{1,\ldots,n\}$ let $A'_i = \sum_{j=1}^d (x_i)_j C_j$. It is easily verified that $A'_i$ is symmetric such that $(A'_i)^2=\Id$, and moreover 
\begin{equation}\label{eq:ac-1a}
\forall i\neq j \in \{1,\ldots,n\}\;,\quad \big(A_i'- A_j'\big)^2 \,=\, \big(2- 2 \,x_i\cdot x_j\big)\Id\;.
\end{equation}
Let $A_i' = P_{i,+}'-P_{i,-}'$ be the spectral decomposition, 
and $X_i = 2^{-d'} P_{i,+}'$, $Y_i = 2^{-d'} P_{i,-}'$. Let $\sigma  = 2^{-d'} \Id$ and $O=0$. 
Then $\|\sigma-O\|_1 = 1$.
Using that $A_i$ has trace $0$, we also have
\[\|X_i-O\|_1 = \|Y_i-O\|_1 = \|\sigma-X_i\|_1=\|\sigma-Y_i\|_1=\frac{1}{2}\;,\]
and $\|X_i-Y_i\|_1=1$, for all $i\in\{1,\ldots,n\}$. 
It only remains to consider the distance between different $i$ and $j$. 
Using that $X_i-X_j = 2^{-d'-1} (A'_i-A'_j)$, the condition $|x_i\cdot x_j|\leq \delta/4$ for $i\neq j$, and~\eqref{eq:ac-1a}, it follows that 
\[\big(1-\frac{\delta}{4}\big)\frac{\sqrt{2}}{2}\leq \|X_i-X_j\|_1 \leq \big(1+\frac{\delta}{4}\big)\frac{\sqrt{2}}{2}\;.\] Similar bounds hold for pairs of the form $(X_i-Y_j)$ and $(Y_i-Y_j)$. Scaling all operators by $(1-\delta/4)^{-1}$ gives an embedding in $\sch_1^d$ 
with distortion at most
$(1+\delta/4)(1-\delta/4)^{-1} \leq (1+\delta)$. 
\end{proof}

\bibliographystyle{alpha}
\bibliography{connes,dimred_schatten}

\end{document}